\documentclass[12pt]{article}
\usepackage{amssymb, amsmath, latexsym, a4}
\usepackage[latin1]{inputenc}
\usepackage{graphicx}
\usepackage[all]{xy}
\usepackage[usenames]{color}

\newenvironment{proof}[1][Proof]{\noindent\textbf{#1.} }{\ \rule{0.5em}{0.5em}}

\newtheorem{De}{Definition}[section]
\newtheorem{Th}[De]{Theorem}
\newtheorem{Pro}[De]{Proposition}
\newtheorem{Le}[De]{Lemma}
\newtheorem{Co}[De]{Corollary}
\newtheorem{Rem}[De]{Remark}
\newtheorem{Ex}[De]{Example}

\newcommand{\Lieh}{\ensuremath{\mathfrak{h}}}
\newcommand{\Lieg}{\ensuremath{\mathfrak{g}}}

\newcommand{\Liem}{\ensuremath{\mathfrak{m}}}
\newcommand{\LieK}{\ensuremath{\mathcal{K}}}

\newcommand{\LieD}{\ensuremath{\mathfrak{D}}}

\newcommand{\Lien}{\ensuremath{\mathfrak{n}}}

\newcommand{\Leib}{\ensuremath{\mathsf{Leib}}}

\newbox\pullbackbox
\setbox\pullbackbox=\hbox{\xy 0;<1mm,0mm>: \POS(4,0)\ar@{-} (0,0) \ar@{-} (4,4)
\endxy}

\newdir{>>}{{}*!/3.5pt/:(1,-.2)@^{>}*!/3.5pt/:(1,+.2)@_{>}*!/7pt/:(1,-.2)@^{>}*!/7pt/:(1,+.2)@_{>}}
\newdir{ >>}{{}*!/8pt/@{|}*!/3.5pt/:(1,-.2)@^{>}*!/3.5pt/:(1,+.2)@_{>}}
\newdir{ |>}{{}*!/-3.5pt/@{|}*!/-8pt/:(1,-.2)@^{>}*!/-8pt/:(1,+.2)@_{>}}
\newdir{ >}{{}*!/-8pt/@{>}}
\newdir{>}{{}*:(1,-.2)@^{>}*:(1,+.2)@_{>}}
\newdir{<}{{}*:(1,+.2)@^{<}*:(1,-.2)@_{<}}

  \newcommand{\eh}{\frak h}

\begin{document}

\centerline{\bf On central characteristic ideals and
quasi-Noetherian Leibniz algebras.}

\bigskip
\centerline{\bf Bell Bogmis N. G$^{1a}$, C. Tcheka$^{\ast 1b},$ G.
R. Biyogmam$^{2}$ .}

\bigskip

\centerline{$^{1a}$ Department of Mathematics, Faculty of
    science-University of Dschang} \centerline{Campus Box (237)67
    Dschang, Cameroon} \centerline{ {E-mail address}:  bellnarcisse3@gmail.com}

\bigskip

\centerline{$^{1b}$ Department of Mathematics, Faculty of
science-University of Dschang} \centerline{Campus Box (237)67
Dschang, Cameroon} \centerline{ {E-mail address}:
calvin.tcheka@univ-dschang.org}

\bigskip

\centerline{$^{2}$ Department of Mathematics, Georgia College \&
State University} \centerline{Campus Box 17 Milledgeville, GA
31061-0490} \centerline{ {E-mail address}: guy.biyogmam@gcsu.edu}

\bigskip
\date{}

\bigskip \bigskip \bigskip

\centerline{\bf Abstract} In this paper, we define on one hand, the
notions of characteristics as well as central characteristics ideals
of a given Leibniz algebra $\Lieg$ and provide a necessary condition
under which for two given subalgebras $\emptyset \neq J \subseteq K$
of $\Lieg$, $J$ is a central characteristics two-sided ideal of $K$.
On the other hand, we introduce the class of quasi-Noetherian
Leibniz algebras. This generalizes both the class  of Noetherian
Leibniz algebras and that of quasi-Noetherian Lie algebras
introduced in \cite{Falih A. M. Aldosray}. We provide a necessary
condition for a Leibniz algebra to be quasi-Noetherian. As in the
case of Lie algebras, quasi-Noetherian Leibniz algebras are shown to
be closed under quotients, but not under extensions. Finally, we
leverage the maximal condition of abelian ideals to provide a
characterization of Noetherian Leibniz algebras.
\bigskip

$\ast$Corresponding author

\section{Introduction and Preliminaries}

The concept of Leibniz algebra was introduced in papers published in
the sixties  by Bloh \cite {bl}, and was popularized three decades
later by  Jean Louis Loday \cite{Lo 1, CCC}. These algebras
generalize Lie algebras. So, a lot of research in Leibniz algebras
investigates  analogous results in the category of Lie algebras.

The concept of Noetherian algebras  play an important role on the theories of infinite dimensional  algebras.
This concept along with some generalizations have been investigated by several authors in the cases of
Lie algebras \cite{BPP, Falih A. M. Aldosray,ASN} and rings \cite{DDD}.  As expected, a Noetherian Leibniz
algebra is a Leibniz algebra that satisfies the ascending chain condition on left and right ideals.
 Our aim in this paper is to introduce the class of quasi-Noetherian Leibniz algebras and
 investigate analogue results as in Lie algebras and other algebraic structures.
 The proposed definition in this paper generalizes both Noetherian Leibniz algebras and quasi-Noetherian
  Lie algebras introduced in \cite{Falih A. M. Aldosray}.  This class  contains solvable Leibniz algebras. 
The paper is organized  as follows: For the remaining of this
section, we recall some definitions and  background results needed
in this study. We introduce the notions of characteristic ideals and
central characteristics ideals in the category of Leibniz algebras
and establish some important results in section 2. In section 3, we
define quasi-Noetherian Leibniz algebras, provide some examples and
prove the main results of the paper. In particular,  we provide
several necessary conditions for a Leibniz algebra to be
quasi-Noetherian. We also provide a characterization of
quasi-Noetherian Leibniz algebras and prove that this class of
Leibniz algebras is closed under quotients but not under extensions.
Section 4 provides a condition under which a quasi-Noetherian
Leibniz algebra is Noetherian.


For the notation and terminology defined in this paper, the standard
reference is [8]. For the sake of convenience,
we summarize the main concepts needed in this work.\\

Let $\mathbb{K}$ be a fixed ground field. Throughout the paper, all vector spaces and tensor products are considered over $\mathbb{K}$.

    Recall that a \emph{Leibniz algebra} \cite{CCC} is a vector space ${\Lieg}$  equipped with a bilinear map $[-,-] : \Lieg \otimes \Lieg \to \Lieg$, usually called \emph{Leibniz bracket} of ${\Lieg}$,  satisfying the left \emph{Leibniz identity}:
    \[
    [x,[y,z]]= [[x,y],z]+[y,[x,z]],
    \]
    for all $x, y, z \in \Lieg.$\\
A subalgebra ${\eh}$ of a Leibniz algebra ${\Lieg}$ is said to be a \emph{left (resp. right) ideal} of ${\Lieg}$ if $ [h,g]\in {\eh}$  (resp.  $ [g,h]\in {\eh}$), for all $h \in {\eh}$, $g \in {\Lieg}$. We write ${\eh} \lhd_{l} \Lieg$ (resp. ${\eh} \lhd_{r} \Lieg$ ) if ${\eh}$ is a left (resp.  right) ideal of $\Lieg.$ If ${\eh}$ is both
left and right ideal, then ${\eh}$ is called a \emph{two-sided ideal} of ${\Lieg}$ (or simply an ideal of $\Lieg$), and we write ${\eh} \lhd \Lieg$ .
 In this case, $\Lieg/\Lieh$ naturally inherits a Leibniz algebra structure. Recursively  ${\eh} \lhd^{n} \Lieg $ if ${\eh} \lhd J \lhd^{n-1} \Lieg $,
 and ${\eh}$ is said to be a $n$-step two-sided subideal of the Leibniz algebra $\Lieg.$ ${\eh}$ is a subideal (written ${\eh}$ si $\Lieg$)
 if it is a $n$-step two-sided subideal of $\Lieg,$ for some $ n $. \\
 \indent Let $\Lieg$ be a Leibniz algebra and let $I,J \lhd \Lieg$.
Then  $$[[I,I],J] \subseteq [I,[I,J]].$$
       Denote by $ \Leib(\Lieg)$ the
subspace of ${\Lieg}$ spanned by all elements of the form $[x,x]$,
$x \in \Lieg$. Since $ \Leib(\Lieg)$ is a two-sided ideal of $\Lieg,$
often referred to as the Leibniz kernel of $\Lieg,$ then the quotient ${\Lieg}_ {_{\rm Lie}}=\Lieg/ \Leib(\Lieg)$ is a Lie algebra,  referred to as the  Liezation of $\Lieg$.\\
The left-center of a Leibniz algebra $\Lieg$ is the two-sided ideal $$Z^l(\Lieg) = \{ x \in \Lieg\mid [x,y]=0 ~\text{for all}~y \in \Lieg \},$$
while the right-center of a Leibniz algebra $\Lieg$ is the set $$Z^r(\Lieg) = \{ x \in \Lieg\mid [y,x]=0 ~\text{for all}~y \in \Lieg \}.$$ So,
$Z(\Lieg) = Z^l(\Lieg) \cap Z^r(\Lieg)$ is called the center of $\Lieg.$ It is a two-sided ideal of $\Lieg$.\\
The lower central series of $\Lieg$ is
the sequence of two-sided ideals of $\Lieg$ defined inductively by
\[ \cdots \trianglelefteq {\gamma_i(\Lieg)} \trianglelefteq \cdots
\trianglelefteq \gamma_2(\Lieg)  \trianglelefteq {\gamma_1(\Lieg)}=
{\Lieg} \quad \text{and} \quad \gamma_i(\Lieg) =[{\Lieg},
\gamma_{i-1}(\Lieg)], \quad   i \geq 2.
\]
 The derived series $ \Lieg^{(n)} $ is defined recursively by
\begin{equation*}
\begin{aligned}
\Lieg^{(0)}& = \Lieg ~~ ~~ ~~ ~~ and ~~ \Lieg^{(n+1)} & = [\Lieg^{(n)}, \Lieg^{(n)}]~,~~~~ n \geq 0.
\end{aligned}
\end{equation*}
 \begin{Le}\label{2.1}
    \begin{enumerate}
        \item[(a)] If $\Lieg_{1}$ and $\Lieg_{2}$ are two Leibniz algebras, then
        $(\Lieg_{1} \oplus \Lieg_{2})^{(k)}=\Lieg^{(k)}_{1} \oplus
        \Lieg^{(k)}_{2},$ for all $k\in \mathbb{N}^{\ast}.$
        \item[(b)] If $\Lieg$ is a Leibniz algebra, then $(\Lieg^{(m)})^{(n)}= \Lieg^{(m+n)}$, for all $m, n \in \mathbb{N}$.
    \end{enumerate}
\end{Le}
\begin{proof}
    \begin{enumerate}
        \item[(a)] Consider the respective Leibniz brackets on $\Lieg_{1}$ and $\Lieg_{2}$ defined as follows:
        $$ \begin{array}{ccc}
            \Lieg_{1} \otimes \Lieg_{1}& \stackrel{[-,-]_{\Lieg_{1}}}\longrightarrow & \Lieg_{1}\\
            & (x_{1},y_{1})\longmapsto & [x_{1}, y_{1}]_{\Lieg_{1}}
        \end{array};$$
        $$ \begin{array}{ccc}
            \Lieg_{2} \otimes \Lieg_{2}& \stackrel{[-,-]_{\Lieg_{2}}}\longrightarrow & \Lieg_{2}\\
            & (x_{2},y_{2})\longmapsto & [x_{2}, y_{2}]_{\Lieg_{2}}
        \end{array}.$$
        It is obvious that $ \Lieg_{1}\oplus\Lieg_{2}$ endowed with the
        bracket $[-,-]_{\Lieg_{1}\oplus \Lieg_{2} }$ defined as follows:
        $$ \begin{array}{ccc}
            (\Lieg_{1}\oplus\Lieg_{2}) \otimes (\Lieg_{1}\oplus \Lieg_{2}) & \stackrel{[-,-]_{\Lieg_{1}\oplus \Lieg_{2} }}\longrightarrow & \Lieg_{1}\oplus\Lieg_{2} \\
            & (x_{1}+x_{2},y_{1}+y_{2})\longmapsto & [x_{1}+x_{2},
            y_{1}+y_{2}]_{\Lieg_{1}\oplus \Lieg_{2}}= [x_{1},
            y_{1}]_{\Lieg_{1}}+[x_{2}, y_{2}]_{\Lieg_{2}}
        \end{array}$$
        is a Leibniz algebra. Next, we verify that $(\Lieg_{1} \oplus \Lieg_{2})^{(k)}=\Lieg^{(k)}_{1} \oplus
        \Lieg^{(k)}_{2},$ for all $k\in \mathbb{N}^{\ast}.$ We proceed
        by induction on $k.$ For $k= 1$, we have $$\begin{array}{ccc}(\Lieg_{1} \oplus \Lieg_{2})^{(1)}&=&[\Lieg_{1} \oplus \Lieg_{2}, \Lieg_{1} \oplus \Lieg_{2}]_{\Lieg_{1}\oplus\Lieg_{2}}\\
            &=& [\Lieg_{1}, \Lieg_{1}]_{\Lieg_{1}} \oplus [\Lieg_{2}, \Lieg_{2}]_{\Lieg_{2}}\\
            &=& \Lieg^{(1)}_{1} \oplus \Lieg^{(1)}_{2}.\end{array}$$
            Assume that the identity holds for all integer $r<
        k$ and let's show it also holds for $k.$ $$\begin{array}{ccc}(\Lieg_{1} \oplus \Lieg_{2})^{(k)}&=&[(\Lieg_{1} \oplus
            \Lieg_{2})^{(k-1)}, (\Lieg_{1} \oplus \Lieg_{2})^{(k-1)}]_{\Lieg_{1}\oplus \Lieg_{2} }\\
            &=&[\Lieg_{1}^{(k-1)} \oplus
            \Lieg_{2}^{(k-1)}, \Lieg_{1}^{(k-1)} \oplus \Lieg_{2}^{(k-1)}]_{\Lieg_{1}\oplus \Lieg_{2} }\\
            &=& [\Lieg_{1}^{(k-1)}, \Lieg_{1}^{(k-1)}]_{\Lieg_{1}} \oplus [\Lieg_{2}^{(k-1)}, \Lieg_{2}^{(k-1)}]_{\Lieg_{2}}\\
            &=& \Lieg^{(k)}_{1} \oplus \Lieg^{(k)}_{2}.\end{array}$$ So for all
        $k\in \mathbb{N},$ $(\Lieg_{1} \oplus \Lieg_{2})^{(k)}=\Lieg^{(k)}_{1} \oplus
        \Lieg^{(k)}_{2}.$
        \item [$(b)$] We proceed by induction on $n$ while fixing $m$.\\
        For $n= 0$, $(\Lieg^{(m)})^{(0)}= \Lieg^{(m)};$  for $n= 1,$ $\Lieg^{(m+1)}=[\Lieg^{(m)},\Lieg^{(m)}]= (\Lieg^{(m)})^{(1)}.$\\
        Assume that the identity holds for all $k< n$ and let us show it is true for
        $n.$
        $$\begin{array}{ccc}
            \Lieg^{(m+n)}& =& [\Lieg^{(m +
                n-1)},\Lieg^{(m + n-1)}]\\
            &=&[(\Lieg^{(m)})^{(n-1)},(\Lieg^{(m)})^{(n-1)}]\\
            &=& (\Lieg^{(m)})^{(n)}.
        \end{array}$$
    \end{enumerate}
    Thus for all $m, n \in \mathbb{N}, \Lieg^{(m+n)}
    =(\Lieg^{(m)})^{(n)}$.
\end{proof}\\

\begin{De}\cite{A. A. O.}
A Leibniz algebra $\Lieg$ is said to be simple if its Liezation is a
${simple}$ Lie algebra and its Leibniz kernel
$Leib(\Lieg)$ is a simple ideal. Equivalently, $\Lieg$ is
${simple}$ if and only if its Leibniz kernel
$Leib(\Lieg)$ is the only non-trivial ideal of $\Lieg$.
\end{De}

\begin{De}\cite{D. K. M.}
    A Leibniz algebra $\Lieg$ is said to be
    ${nilpotent}$ of class $c$ (respectively ${solvable}$  of derived
    length $ \leq n $) if\ $\gamma_{c+1}(\Lieg) = 0$ and
    $\gamma_c(\Lieg) \neq 0$ (resp. $ \Lieg^{(n)} = 0
    $).

\end{De}
\indent Let $ \Liem$ and $\Lien$ be two non-empty subsets of $\Lieg,$ denote by $ \langle \Liem\rangle $ the subalgebra
generated by $\Liem$ and define the centralizer of $\Liem$ and $\Lien$ over $\Lieg$ by $C_{\Lieg}(\Liem,  \Lien)=\{a\in \Lieg:[a, b]\in \Lien ~ \mbox{for all} ~ b\in \Liem\}.$
 So the left center of $\Lieg$ is $ C_{\Lieg}(\Lieg, 0)= Z^{l}(\Lieg)$.

\indent The closure operators $Q$ and $ E $ are defined as
follows: Let $ \mathfrak{X} $ be a class of Leibniz algebras:
\begin{equation*}
\begin{aligned}
Q\mathfrak{X} &= \{\Lieg/\eh : \eh\lhd \Lieg ~ \& ~ \Lieg \in \mathfrak{X}\};\\
E\mathfrak{X} &= \{\Lieg : \exists \eh, \eh\lhd \Lieg ~\&  ~\eh, \Lieg/\eh \in \mathfrak{X}\}.
\end{aligned}
\end{equation*}

 If $ \mathfrak{X}, \mathfrak{Y} $ are two
classes of Leibniz algebras, then the class $
\mathfrak{X}\mathfrak{Y} $ of $\mathfrak{X}$-by-$\mathfrak{Y} $
Leibniz algebras consists of all Leibniz algebras $\Lieg$ such that there exists $\eh\in \mathfrak{X}$ such that $ \eh \lhd \Lieg, ~ \Lieg/\eh \in \mathfrak{Y} $.
\begin{Rem}
 In the category of Leibniz algebras, $\mathfrak{X}\mathfrak{X}=E\mathfrak{X}$.
\end{Rem}
\indent For the remainder of the paper, we use the following notations for these classes of Leibniz algebras.
\begin{equation*}
\begin{aligned}
\mathfrak{F} & =~ finite\mbox{-}dimensional\\
\mathfrak{U} & = ~ abelian\\
\mathfrak{N} & = ~nilpotent\\
\mathfrak{N}_{c} &=~ nilpotent ~ of ~ class \leq c\\
E\mathfrak{U} &= ~ solvable\\
\mathfrak{U}^{d} &=~ solvable ~ of ~ derived ~ length \leq d.
\end{aligned}
\end{equation*}
Any other notation will be introduced as needed.

 \section{Characteristic Ideal in Leibniz Algebras}
In this section, we introduce  the
notion of characteristic ideal in the category of Leibniz algebras.
\begin{De}\cite{BT}  A derivation of $\Lieg$ is a $
    \mathbb{K}$-linear map $\delta : \Lieg \to \Lieg$ such that for all
    $ x,y \in \Lieg$, $ \delta([x,y]) = [\delta(x),y] + [x,\delta(y)]~.$
\end{De}

For a Leibniz algebra $\Lieg$ and $a \in \Lieg$, denote the left multiplication operator by
$ {L_{a}=[a, -]}: \Lieg \longrightarrow  \Lieg$ defined by
$L_{a}(x)=[a, x],$
and the right multiplication operator ${R_{a}=[- , a]}: \Lieg\longrightarrow  \Lieg$ defined by
$ R_{a}(x)=[x, a].$ It is obvious that the right multiplication operator $R_{a}$ is not a derivation while the left multiplication operator $L_{a}$ is a
derivation.

\begin{De}
    Let $\Lieg$ be a Leibniz algebra. A two-sided ideal $ \eh$ of
    $\Lieg$ is said to be characteristic, written $\eh$ $ch$ $\Lieg$, if
    it is invariant under all derivations of $\Lieg$.
\end{De}

\begin{De}
    Let $\Lieg$ be a Leibniz algebra. A two-sided ideal $ \eh$ of
    $\Lieg$  is said to be central characteristic, written $\eh$
    $ch^{z}$ $\Lieg$, if it is invariant under all central derivations
    of $\Lieg$.
\end{De}

\begin{Rem}\label{3.4}
    \begin{enumerate}
        \item It is obvious that any characteristic ideal of $\Lieg$ is a central
        characteristic ideal of $\Lieg$.
        \item If $\eh$ $ch$ $J \lhd \Lieg$, then for all $g\in\Lieg,$ $L_{\Lieg}\in Der(\Lieg)$ and the restriction $L_{\Lieg}|_{J}\in Der(J).$ Since $\eh$ $ch$ $J,$ then for all $h\in\eh,$ $[g,h]\in\eh.$ Therefore    $\eh \lhd_{l} \Lieg$. If in
        addition, $\eh$ is invariant under all right multiplication
        operators, then $\eh \lhd \Lieg$. In particular, if $\Lieg$ is a Lie
        algebra and $\eh$ ch $J \lhd \Lieg$, then $\eh \lhd \Lieg;$ as proven in  \cite{Falih A. M. Aldosray}.
    \end{enumerate}
\end{Rem}

\begin{De}
    If $\Lieg$ is a Leibniz algebra having a two-sided ideal $I$ and a
    subalgebra $J$ such that $\Lieg= I + J$ and $I \cap J= \{0\}, $ then
    $\Lieg$ is said to be a split extension of $I$ by $J$.
\end{De}
\indent Denote by $\frak{S}_{\frak{E}}$ the set of all Leibniz
algebras which are split extensions. The class of Leibniz algebras
of type $T_{Lie}$ recently introduced in
\cite{KTB}, is a subclass of $\frak{S}_{\frak{E}}.$
\begin{Pro}
    Let $\Lieg= I + J$ be a split extension. Then there exists a homomorphism
    of Leibniz algebras
    $$\begin{array}{ccc} J & \stackrel{^{I}\theta_{J}}\longrightarrow &
        \mbox{Der}(I)\\
        a & \longmapsto & ^{I}\theta_{J}(a)= L_{a}
    \end{array}.$$ Moreover, there is a one-to-one correspondence between split
    extensions and the set of homomorphisms $^{I}\theta_{J}.$
\end{Pro}
\begin{proof}
    Clearly, it is obvious that
    $^{I}\theta_{J}$ is a $\mathbb{K}$-linear map. To verify that
    $^{I}\theta_{J}$ is compatible with the brackets, consider $a, b \in
    J$ and $x\in I.$ One has $$\begin{array}{ccc}
        L_{[a, b]}(x)&=& [[a, b], x]\\
        &=& [a, [b, x]] -  [b, [a, x]]\\
        &=& L_{a}\circ L_{b}(x) - L_{b}\circ L_{a}(x)\\
        &=& [L_{a}, L_{b}](x).
    \end{array}$$
    Thus $^{I}\theta_{J}$ is a homomorphism of Leibniz algebras. \\Now, set $\LieD = \{ ^{I}\theta_{J}, I + J ~ ~ \mbox{split
        extension} \}.$
    Since $^{I}\theta_{J} = ~^{I^{\prime}}\theta_{J^{\prime}}$ implies that $J = J^{\prime}$, for
    some split extensions $\Lieg=I + J=I^{\prime} + J^{\prime}$ with $I \cap J=I^{\prime} \cap J^{\prime}=\{0\};$
    therefore the map which assigns to any split extension $I + J$ the
    homomorphism of Leibniz algebra $^{I}\theta_{J}$  is injective.
    Moreover, the map is surjective by construction.
\end{proof}
\begin{Pro}
    Let $\{0\}\neq J \subseteq \LieK$ be two subalgebras of a Leibniz
    algebra $\Lieg$ and assume that $J \lhd \Lieg$ whenever $\LieK \lhd
    \Lieg.$ Then $J$ is central characteristic in $\LieK.$
\end{Pro}
\begin{proof}
    Let $d$ be any central derivation of $\LieK$ and consider the
    split extension $\LieK \oplus <d>$ where $<d>$ denotes the Lie
    subalgebra of $\mbox{Der(\LieK)}$ spanned by the central derivation
    $d.$ It is obvious that $\LieK \oplus <d>$ endowed with the bracket
    defined by $[k + \alpha d, k^{\prime}+ \beta d]= [k, k^{\prime}] +
    \beta d(k)- \alpha d(k^{\prime})$ is a Leibniz algebra; where $k,
    k^{\prime} \in \LieK$ and $\alpha, \beta\in \mathbb{K}.$ Now notice
    that $\LieK$ is a two-sided ideal of $\LieK \oplus <d>,$ so by hypothesis, $J$ is also a two-sided ideal of $\LieK \oplus <d>$ and
    we have $[J, d]\subseteq J$ and $[d, J]\subseteq J.$ Moreover, $[d,
    J]= -d(J)$ and $[J, d]= d(J).$ Hence $J$ is central
    characteristic in $\LieK.$
\end{proof}

\section{Quasi-Noetherian Leibniz Algebras}
In this section, we introduce and study the concept of quasi-Noetherian Leibniz algebras.
 \subsection{Definitions  and Properties}

\begin{De}
    A Leibniz algebra $\Lieg$ is said to be left (resp. right) \textit{ quasi-Noetherian} if for every ascending chain of ideals
    \begin{equation*}
    I_{0}\subseteq I_{1} \subseteq I_{2} \subseteq \cdots,
    \end{equation*}
    there exists $ m_{l} ~ (\mbox{resp.} ~ m_{r}) \in \mathbb{N} $ such that
$
    \left[\Lieg^{(m_{l})}, \bigcup_{n \in \mathbb{N}}I_{n}\right] \subseteq I_{m_{l} }
$

    (resp.
$
        \left[\bigcup_{n \in \mathbb{N}}I_{n}, \Lieg^{(m_{r})}\right] \subseteq I_{m_{r}}).
$
    We say that a Leibniz algebra $\Lieg$ is \textit{ quasi-Noetherian} if it is both left and right \textit{ quasi-Noetherian}.

\end{De}

\noindent{\bf Notations}\\
The class of left(resp. right)  quasi-Noetherian
Leibniz algebras is denoted  $ q\max$-$\lhd_{l}$ (resp. $
q\max$-$\lhd_{r}$) and that of all quasi-Noetherian Leibniz
algebras is denoted by  $ q\max$-$\lhd. $

\begin{Rem}\label{4.2}

\begin{enumerate}
\item  As in the case of Lie algebras, a Leibniz algebra is said to be Noetherian if any ascending chain of two sided-ideals terminates.
 Thus every Noetherian Leibniz algebra is quasi-Noetherian.
Indeed, since for every ascending chain of ideals
    \begin{equation*}
    I_{0}\subseteq I_{1} \subseteq I_{2} \subseteq \cdots,
    \end{equation*}
$\cup_{n \in \mathbb{N}}I_{n}\subseteq I_{q} $ for some $q\in \mathbb{N}.$ It follows that $ [\Lieg^{(q)}, \cup_{n \in \mathbb{N}}I_{n}]\subseteq I_{q} $ and  $ [\cup_{n \in \mathbb{N}}I_{n}, \Lieg^{(q)}]\subseteq I_{q}.$
  \item Obviously, every solvable Leibniz algebra is quasi-Noetherian.
  \end{enumerate}
\end{Rem}

\indent  In the next results, we provide necessary conditions for a Leibniz algebra to be quasi-Noetherian.

\begin{Th}

    Let $\Lieg$ be a Leibniz algebra over any field $ \mathbb{K} $.
        If for any nonempty collection $ \mathcal{C} $ of ideals of $ \Lieg $, there exists $ I \in \mathcal{C} $ and $ m \in \mathbb{N} $ such that $ [\Lieg^{(m)},J] \subseteq I $ and $ [J, \Lieg^{(m)}] \subseteq I $ for any $ J\in  \mathcal{C} $ with $ I \subseteq J  $, then  $\Lieg$ is a quasi-Noetherian Leibniz algebra.
\end{Th}
\begin{proof}\\
\indent Assume the given condition is satisfied and consider an arbitrary
ascending chain of ideals of $\Lieg$
\begin{equation*}
I_{0}\subseteq I_{1} \subseteq I_{2} \subseteq \cdots .
\end{equation*}
Let $\mathcal{C}:=\{ I_{r}: r\in \mathbb{N}\}$ be the set of
ideals of  the above ascending sequence. From
hypothesis, there exists $m_{0}\in \mathbb{N}$ and an ideal $I_{m_{1}}\in
\mathcal{C}$ such that $[\Lieg^{(m_{0})}, I_{k}] \subseteq
I_{m_{1}}$ for all $I_{k}\in \mathcal{C}$ with $m_{1} \leq k.$
Since the sequence $\{\Lieg^{(s)}= [\Lieg^{(s-1)},
\Lieg^{(s-1)}], s\ge 1\}$ is a descending sequence of ideals
of the Leibniz algebra $\Lieg,$ we have $[\Lieg^{(m)}, I_{k}]
\subseteq I_{m},$ for all $I_{k}\in \mathcal{C}$ with $\max\{m_0,
m_1\} = m  \leq k.$ Thus we finally obtain $[\Lieg^{(m)},
\bigcup_{r\in \mathbb{N}} I_{r}] \subseteq I_{m},$ that is, $\Lieg$
is a left quasi-Noetherian Leibniz algebra. In a similar way,
one can show that $\Lieg$ is also a right quasi-Noetherian
Leibniz algebra. Therefore $ \Lieg \in q \max$-$\lhd $.
\end{proof}

\begin{Pro}
Let $I$ be a two-sided ideal of a Leibniz algebra $\Lieg$. If $
\Lieg/I $ is solvable and the sets of ideals
 $ \{[\mathcal{H},I] : \mathcal{H} \lhd \Lieg\} $ and $ \{[I,\mathcal{H}] : \mathcal{H} \lhd \Lieg\} $ each have  a maximal element,
  then $\Lieg$ is quasi-Noetherian.
\end{Pro}
\begin{proof}\\
Let  $\mathcal{J}_{0} \subseteq \mathcal{J}_{1} \subseteq \mathcal{J}_{2} \subseteq \cdots, $ be
 an ascending sequence of ideals of $\Lieg$. Set $\mathcal{J}=\underset{i \in \mathbb{N}}{\cup}\mathcal{J}_{i}.$
 Consider the sets $\Lieg_{1}=\{[\mathcal{J}_{i}, I], i\in \mathbb{N} \}$ and $\Lieg_{2}= \{[I, \mathcal{J}_{i}], i\in \mathbb{N} \}.$
  Then $\Lieg_{1}\subseteq \{[\mathcal{H},I] : \mathcal{H} \lhd \Lieg\} $ and $\Lieg_{2}\subseteq \{[I,\mathcal{H}]: \mathcal{H} \lhd \Lieg\}.$
   Since  $\{[\mathcal{H},I] : \mathcal{H} \lhd \Lieg\} $ and $ \{[I,\mathcal{H}] : \mathcal{H} \lhd \Lieg\} $ each admit a maximal element,
 so do $\Lieg_{1}$ and $\Lieg_{2}$. Now let $m_{0}, m_{1}\in \mathbb{N}$ such that $[\mathcal{J}_{m_{0}}, I]$ and $[I,  \mathcal{J}_{m_{1}}]$
 are the respective maximal elements of $\Lieg_{1}$ and $\Lieg_{2}.$ Taking $m= \mbox{max}\{m_{0}, m_{1}\},$ we have $[\mathcal{J}, I] \subseteq [\mathcal{J}_{m}, I]$
  and $[I, \mathcal{J}] \subseteq [I,  \mathcal{J}_{m}].$
  Moreover, since $\frac{\Lieg}{I}$ is solvable, there exists $k\in \mathbb{N}$ such that
     $(\frac{\Lieg}{I})^{(k)} =0$. Thus $\Lieg^{(k+1)} \subseteq I$, so $[\mathcal{J},\Lieg^{(k+1)}] \subseteq [\mathcal{J},I]
     \subseteq [\mathcal{J}_{m},I] \subseteq \mathcal{J}_{m}$.
Now, taking $p=\max \{k+1, m \}$, we have $[\mathcal{J},\Lieg^{(p)}] \subseteq \mathcal{J}_{p}$, i.e. $[\underset{i \in \mathbb{N}}{\cup}\mathcal{J}_{i},\,\Lieg^{(p)}] \subseteq \mathcal{J}_{p}.$
    Hence $ \Lieg \in q \max$-$\lhd_r .$
    Similarly, one shows that there exists $s\in \mathbb{N}$ such that $[\Lieg^{(s)},\underset{i \in \mathbb{N}}{\cup}\mathcal{J}_{i}] \subseteq
    \mathcal{J}_{s}$. That is, $\Lieg \in q \max$-$\lhd_{l} $. Therefore $\Lieg \in q \max$-$\lhd.$
\end{proof}\\

\begin{Th}\label{4.5}
    Let $\Lieg$ be a Leibniz algebra, with $\mathcal{K} \lhd \Lieg$. If $\mathcal{K}$ is Noetherian and $\frac{\Lieg}{\mathcal{K}}$ is quasi-Noetherian, then $\Lieg$ is quasi-Noetherian.
\end{Th}
\begin{proof}
Consider $\mathcal{J}_{0} \subseteq \mathcal{J}_{1} \subseteq
\mathcal{J}_{2} \cdots,$ an ascending chain of ideals of $\Lieg$.
Then $\mathcal{J}_{0} \cap \mathcal{K} \subseteq \mathcal{J}_{1}
\cap \mathcal{K} \subseteq \mathcal{J}_{2} \cap \mathcal{K} \cdots$
is an ascending chain of ideals of $\mathcal{K}$ and \\
$\frac{\mathcal{J}_{0} +  \mathcal{K}}{ \mathcal{K}} \subseteq
\frac{\mathcal{J}_{1}+  \mathcal{K}}{ \mathcal{K}} \subseteq
\frac{\mathcal{J}_{2} +  \mathcal{K}}{ \mathcal{K}} \cdots$ is an
ascending chain of ideals of $\frac{\Lieg}{ \mathcal{K}}$. Since $
\mathcal{K}$ is Noetherian, then the ascending chain of ideals $
\{ \mathcal{J}_{s} \cap \mathcal{K}, s\in \mathbb{N}\} $ terminates.
Set $\mathcal{J}_{s_{0}} \cap \mathcal{K}$, its maximal element.
On the other hand since $\frac{\Lieg}{ \mathcal{K}}$ is
quasi-Noetherian, there exists $m\in \mathbb{N}$ such that
$[\Lieg^{(m)},\underset{i \in
\mathbb{N}}{\cup}\mathcal{J}_{i}] + \mathcal{K} \subseteq
\mathcal{J}_{m} + \mathcal{K}$ and  $[\underset{i \in
\mathbb{N}}{\cup}\mathcal{J}_{i},\Lieg^{(m)}] + \mathcal{K}
\subseteq \mathcal{J}_{m} + \mathcal{K}.$
 Also, $[\Lieg^{(m)},\underset{i \in
\mathbb{N}}{\cup}\mathcal{J}_{i}] \subseteq \underset{i \in
\mathbb{N}}{\cup}\mathcal{J}_{i}$ and
      $[\underset{i \in \mathbb{N}}{\cup}\mathcal{J}_{i},\Lieg^{(m)}] \subseteq \underset{i \in
      \mathbb{N}}{\cup}\mathcal{J}_{i}$. Now let  $x\in [\Lieg^{(m)},\underset{i \in
    \mathbb{N}}{\cup}\mathcal{J}_{i}].$ Then $x+ \mathcal{K}=y+\mathcal{K}$ for some $y\in \mathcal{J}_{m},$ that is $x-y\in \mathcal{K}.$ Since $\mathcal{J}_{m}\subseteq \underset{i \in
\mathbb{N}}{\cup}\mathcal{J}_{i} $ and $[\Lieg^{(m)},\underset{i \in
\mathbb{N}}{\cup}\mathcal{J}_{i}] \subseteq \underset{i \in
\mathbb{N}}{\cup}\mathcal{J}_{i},$ it follows that $x-y\in  \underset{i \in
\mathbb{N}}{\cup}\mathcal{J}_{i} $ and thus $x-y\in  \mathcal{K}\cap(\underset{i \in
\mathbb{N}}{\cup}\mathcal{J}_{i} )=\mathcal{K} \cap \mathcal{J}_{s_{0}}.$ So $x\in\mathcal{J}_{m}+\mathcal{K} \cap \mathcal{J}_{s_{0}}.$ This proves that  $[\Lieg^{(m)},\underset{i \in
    \mathbb{N}}{\cup}\mathcal{J}_{i}]\subseteq \mathcal{J}_{m}+\mathcal{K} \cap \mathcal{J}_{s_{0}}.$
    Set $p= \mbox{max}\{m, s_{0}\}$ and as $\{\Lieg^{(k)}, k\in
\mathbb{N}\}$ is a descending sequence, then
$\left[\Lieg^{(p)},\underset{i \in
\mathbb{N}}{\cup}\mathcal{J}_{i} \right] \subseteq \mathcal{J}_{p}.$ Similarly one can show that $[\underset{i \in
\mathbb{N}}{\cup}\mathcal{J}_{i},\Lieg^{(q)}] \subseteq
\mathcal{J}_{q}$ for some $q\in \mathbb{N}.$ Hence
 $\Lieg$ is quasi-Noetherian.
\end{proof}
\begin{Co}\label{4.6}
    Let $\Lieg$ be a Leibniz algebra, with $I \lhd \Lieg$.
    \begin{enumerate}
        \item [$(a)$] If $I$ is simple and $\frac{\Lieg}{I}$ is solvable, then $\Lieg$ is quasi-Noetherian.
        \item [$(b)$] If $I$ is simple and $\frac{\Lieg}{I}$ is abelian, then $\Lieg$ is quasi-Noetherian..
    \end{enumerate}
\end{Co}
\begin{proof}
 To prove $(a)$, let $\Lieg$ be a Leibniz algebra and $I$ a simple ideal of $\Lieg$
        such that $\frac{\Lieg}{I}$ is a solvable Leibniz
        algebra. Since $I$ is simple,  it strictly contains only one non-trivial subideal. So any ascending chain of subideals of $I$ terminates.
    So $I$ is Noetherian. Moreover,  by Remark \ref{4.2}, $\frac{\Lieg}{I}$ is solvable, and thus it is quasi-Noetherian.
        The result follows by Theorem \ref{4.5}.
 $(b)$ is a consequence of $(a)$ since if $\frac{\Lieg}{I}$ is abelian, then $\frac{\Lieg}{I}$ is
        solvable.
\end{proof}

 In the next results, we establish some useful properties about closure.
\begin{Pro}
    The class $ q \max $-$\lhd  $ of Leibniz algebras is $ Q$-closed but not $E$-closed.
\end{Pro}

\begin{proof}
First, we show that  the class $ q \max $-$\lhd  $  is $ Q$-closed. Indeed, let $\Lieg$ be a quasi-Noetherian Leibniz algebra and $I$ an ideal of $\Lieg.$
         We show that the quotient Leibniz algebra  $\frac{\Lieg}{I}$ is also quasi-Noetherian.
        Consider an arbitrary ascending chain of ideals of  $\frac{\Lieg}{I}:$ $\frac{\mathcal{J}_{0}}{I}
        \subseteq \frac{\mathcal{J}_{1}}{I} \subseteq \frac{\mathcal{J}_{2}}{I} \subseteq \cdots$,
        from which is obtained the following ascending sequence of ideals of $\Lieg:$ $\mathcal{J}_{0}
        \subseteq \mathcal{J}_{1} \subseteq \mathcal{J}_{2}\subseteq \cdots.$
        Now since $\Lieg$ is a quasi-Noetherian Leibniz algebra, there exists $m\in \mathbb{N}$ such
         that $[\Lieg^{(m)},\underset{s \in \mathbb{N}}{\cup}\mathcal{J}_{s}] \subseteq \mathcal{J}_{m}$
         and $[\underset{s \in \mathbb{N}}{\cup}\mathcal{J}_{s},\Lieg^{(m)}] \subseteq \mathcal{J}_{m}.$ So there exists
         $m\in \mathbb{N}$ such that  $[\underset{s \in \mathbb{N}}{\cup}\frac{\mathcal{J}_{s}}{I}, (\frac{\Lieg}{I})^{(m)}] =
         [\frac{\underset{s \in \mathbb{N}}{\cup}\mathcal{J}_{s}}{I}, (\frac{\Lieg}{I})^{(m)}] \subseteq \frac{[\underset{s
         \in \mathbb{N}}{\cup}\mathcal{J}_{s}, \Lieg^{(m)}]}{I} \subseteq \frac{\mathcal{J}_{m}}{I}$ and $[(\frac{\Lieg}{I})^{(m)},
          \underset{s \in \mathbb{N}}{\cup}\frac{\mathcal{J}_{s}}{I}] = [(\frac{\Lieg}{I})^{(m)}, \frac{\underset{s \in
          \mathbb{N}}{\cup}\mathcal{J}_{s}}{I}] \subseteq \frac{[\Lieg^{(m)}, \underset{s \in \mathbb{N}}{\cup}\mathcal{J}_{s}]}{I}
          \subseteq \frac{\mathcal{J}_{m}}{I}.$  Hence $\Lieg/I \in q \max$-$\lhd$ and therefore the class $ q \max $-$\lhd  $ is $ Q$-closed.\\
Now, the class $ q \max $-$\lhd  $  is not $ E$-closed since this statement fails for Lie algebras. Counterexamples are provided in Section 2 of  \cite{ASN}.
\end{proof}

\begin{Rem}
 A finite direct sum of quasi-Noetherian Leibniz algebras is quasi-Noetherian.
\end{Rem}
\begin{proof}
    Let $\Lieg_{1},\Lieg_{2} \in q \max$-$\lhd$ and let
        $J_{0} \subseteq J_{1} \subseteq J_{3} \subseteq \cdots, $
 be an ascending chain of ideals of $\Lieg_{1} \oplus \Lieg_{2}.$ Observe that for all $r\geq 0,$
    there exists $J_{r}^{1}$ and $J_{r}^{2}$ ideals of $\Lieg_{1}$ and $\Lieg_{2}$ respectively such that $J_{r}= J_{r}^{1} \oplus J_{r}^{2}$.
   Since $\Lieg_{1}$ and $\Lieg_{2}$ are  quasi-Noetherian Leibniz algebras,
   there exists $k_{1},k_{2} \in \mathbb{N}$ such that
    \begin{center}
        $[(\Lieg_{1})^{(k_{1})},\underset{r \in \mathbb{N}}{\bigcup}J^{1}_{r}] \subseteq J^{1}_{k_{1}}$ and
        $[\underset{r \in \mathbb{N}}{\bigcup}J^{2}_{r},(\Lieg_{2})^{(k_{2})}] \subseteq J^{2}_{k_{2}}$.
    \end{center}
     Taking $k= \max(k_{1},k_{2})$, we have  $\Lieg^{(k)} \subset
     \Lieg^{(k_{1})}$ and $\Lieg^{(k)} \subset
     \Lieg^{(k_{2})}.$ Since $((\Lieg_{1}) \oplus (\Lieg_{2}))^{(k)}= (\Lieg_{1})^{(k)}
     \oplus (\Lieg_{2})^{(k)}$ by Lemma \ref{2.1}, we deduce the following:

    \begin{equation*}
    \begin{aligned}
     \left[((\Lieg_{1}) \oplus (\Lieg_{2}))^{(k)},\underset{r \in \mathbb{N}}{\bigcup} J_{r}\right]  &
     = \left[((\Lieg_{1}) \oplus (\Lieg_{2}))^{(k)},
     \underset{r \in \mathbb{N}}{\bigcup}(J^{1}_{r} \oplus
    J^{2}_{r})\right]\\
    & = \left[(\Lieg_{1})^{(k)} \oplus (\Lieg_{2})^{(k)},\underset{r \in \mathbb{N}}{\bigcup}J^{1}_{r} \oplus \underset{r \in \mathbb{N}}{\bigcup}J^{2}_{r}\right]\\
    & = \left[(\Lieg_{1})^{(k)},\underset{r \in \mathbb{N}}{\bigcup}J^{1}_{r}\right] \oplus \left[(\Lieg_{2})^{(k)},\underset{r \in \mathbb{N}}{\bigcup}J^{2}_{r}\right]\\
    & \subseteq \left[(\Lieg_{1})^{(k_{1})},\underset{r \in \mathbb{N}}{\bigcup}J^{1}_{r}\right] \oplus \left[(\Lieg_{1})^{(k_{2})},\underset{r \in \mathbb{N}}{\bigcup}J^{2}_{r}\right]\\
    & \subseteq J^{1}_{k_{1}} \oplus J^{2}_{k_{2}}\\
    & \subseteq J^{1}_{k} \oplus J^{2}_{k} = J_{k}.
    \end{aligned}
    \end{equation*}
    That is, $(\Lieg_{1}) \oplus (\Lieg_{2})$ is a left quasi-Noetherian
    Leibniz algebra.
    In a similar way one can show that $(\Lieg_{1}) \oplus (\Lieg_{2})$ is a right quasi-Noetherian
    Leibniz algebra.
    Thus, $\Lieg_{1} \oplus \Lieg_{2} \in q \max$-$\lhd$.
\end{proof}

\subsection{Examples of quasi-Noetherian Leibniz Algebras}
 The following Lemma is useful
in verifying the Leibniz identity  using a minimum
number of steps.
\begin{Le}\cite[Lemma 1]{C. C.}\label{1.7}
    Let $\Lieg$ be a $\mathbb{K}$-vector spaces endowed with a bilinear map $[-,-],$ call it product.
     Assume that the subspace generated by $[x,x]$ cancels the product to the right.
     In such $\mathbb{K}$-vector spaces, the Leibniz identity is true for the triple $(x,y,z)$ if and only if it is for $(x,z,y)$.
\end{Le}
The following is an example of a quasi-Noetherian Leibniz algebra that is also Noetherian.
\begin{Ex}

\item Let $\Lieg$ be a $\mathbb{K}$-vector space spanned by
$<\{e_{1}, e_{2}, e_{3}, e_{4}, e_{5}, e_{6}\}>.$ Define the bracket
$\Lieg\otimes \Lieg\stackrel{[-, -]}\longrightarrow \Lieg$ as
follows $[e_{2}, e_{2}]= e_{1}, [e_{3}, e_{3}]= e_{5}, [e_{3},
e_{4}]= e_{6}, [e_{4}, e_{3}]= e_{5}, [e_{5}, e_{3}]= e_{6}.$
Using Lemma \ref{1.7}, one can verify that the above bracket satisfies the
Leibniz identity. Since $\Lieg$ is  a finite dimensional Leibniz algebra,
 it is  Noetherian, and thus quasi-Noetherian. Alternatively, set $I$ and $J$ be the subspaces of
$\Lieg,$ respectively spanned by $\{e_{1}, e_{2}\}$ and $\{
e_{3}, e_{4}, e_{5}, e_{6}\}$. Then $I$ and $J$ are two-sided ideals of
$\Lieg.$ Moreover, $I$ is a simple ideal of $\Lieg$ and
$\frac{\Lieg}{I} \cong J$ is a nilpotent Leibniz algebra. Let c be
the class of the nilpotency of $J.$ Since $\gamma_{c}(J)\subseteq
Z^{l}(J)\cap Z^{r}(J)= Z(J)$ (see \cite{D. K. M.}, proposition 4.2
and corollary 4.3), it follows that
 $J$ is a nilpotent (and thus solvable) Leibniz algebra. So by Corollary \ref{4.6} above
 we obtain that $\Lieg$ is a quasi-Noetherian Leibniz algebra.

\end{Ex}

Hereafter we provide an example of a quasi-Noetherian Leibniz algebra that is not Noetherian.

\begin{Ex} Consider the vector space
$\Lieg$ be spanned by $\{e_{1},e_{2}, \cdots\}$ and define on
$\Lieg$ the bracket as follows: $[e_{2},e_{2}]=e_1$ and
$[e_{i},e_{3}]= e_{i+1},$ $i\geq 4$. Again, using Lemma \ref{1.7}, one can verify that the above bracket satisfies the
Leibniz identity, and thus $\Lieg$ is a Leibniz algebra. Next let $I$ and $J$ be the subspaces of $\Lieg$ spanned by $\{e_{1},e_{2}\}$ and
$\{e_{3},e_{4},\cdots\}$ respectively. It is obvious that $I$ and
$J$ are two sided ideals of $\Lieg$. Furthermore $I$ is a simple
ideal while $J$ is a solvable ideal of $\Lieg$. Now since
$\frac{\Lieg}{I}\cong J,$ it follows by corollary \ref{4.6} that $\Lieg$
is a quasi-Noetherian Leibniz algebra.

\end{Ex}

\section{Maximal Condition For Abelian Ideals}
Recall that $\max $-$\lhd \mathfrak{U} $ denotes the class of all
Leibniz algebras satisfying the maximal condition for abelian
ideals; $\mathfrak{U}^{k}$ denotes the class of solvable
Leibniz algebras of derived length $\leq k$ and $\max $-$\lhd
\mathfrak{U}^{k}$ denotes the class of Leibniz algebras
satisfying the maximal condition for $\mathfrak{U}^{k}$ ideals
of $\Lieg.$ Thus $\Lieg \in \max $-$\lhd \mathfrak{U}^{k}$ means
that any ascending chain of $\mathfrak{U}^{k}$ ideals of $\Lieg$ terminates.
In this section, we aim
to relate the class $\max $-$\lhd \mathfrak{U}$ to the classes $ q \max $-$\lhd$ and $\max $-$\lhd$.\\
\indent For any class $ \mathfrak{X} $, set $\mathfrak{X}^{Q} =
\{\Lieg \in \mathfrak{X}: \Lieg/I \in \mathfrak{X}~ \text{for all }~ I \lhd \Lieg\}.$
Notice that $ \mathfrak{X}^{Q}\subseteq \mathfrak{X}.$ For instance, $(Leib)^{Q}=Leib.$\\
\indent The following result leads to a
characterization of Noetherian Leibniz algebras.

\begin{Le}\label{5.1}  
    For any $ k \geq 1, $
\begin{center}
    $(\max $-$\lhd \mathfrak{U})^{Q} = (\max $-$\lhd \mathfrak{U}^{k})^{Q}.$
\end{center}
\end{Le}

\begin{proof}
Let $\Lieg$ be a Leibniz algebra in $ (\max $-$\lhd
\mathfrak{U}^{k})^{Q}$ and consider an arbitrary ascending
chain of
$\mathfrak{U}$ ideals of the quotient Leibniz algebra $\frac{\Lieg}{I}$: \\
\begin{equation*}
    \mathcal{K}_{0} \subseteq \mathcal{K}_{1} \subseteq \cdots \subseteq \mathcal{K}_{r} \subseteq
    \cdots,
\end{equation*}
for a given ideal $I$ of $\Lieg.$ Set $\mathcal{K} = \underset{i \in \mathbb{N}}{\cup} \mathcal{K}_{i}$. \\
 Since for all $i \in
\mathbb{N},$ $\mathcal{K}_{i}$ is an abelian ideal,
$\mathcal{K}^{(1)}=[\mathcal{K},\mathcal{K}]=0$, and
thus $\mathcal{K}^{(k)}=0.$ So for all $i \in \mathbb{N},$
$\mathcal{K}_{i}$ is also a $\mathfrak{U}^{k}$ ideal of
$\frac{\Lieg}{I}$. But $\Lieg \in \max $-$\lhd \mathfrak{U}^{k}$
by hypothesis and $ \mathcal{K}_{0} \subseteq \mathcal{K}_{1}
\subseteq \cdots \subseteq \mathcal{K}_{r} \subseteq \cdots $ is
also an ascending chain of $\mathfrak{U}^{k}$ ideals of
$\frac{\Lieg}{I},$ hence it terminates. Therefore $\Lieg \in (\max
$-$\lhd \mathfrak{U})^{Q}.$ This proves that $(\max $-$\lhd
\mathfrak{U}^{k})^{Q} \subseteq
(\max $-$\lhd \mathfrak{U})^{Q}$.\\
Now, to prove that $(\max $-$\lhd \mathfrak{U})^{Q}
\subseteq (\max $-$\lhd \mathfrak{U}^{k})^{Q},$
we proceed by induction on $k$.
For $k=1$, we have
    $(\max $-$\lhd \mathfrak{U})^{Q} = (\max$-$\lhd \mathfrak{U}^{k})^{Q}.$\\
Hereafter, we assume this inclusion holds for $0 \leq s\leq k$, that
is $(\max $-$\lhd \mathfrak{U})^{Q} \subseteq (\max $-$\lhd
\mathfrak{U}^{s})^{Q},$ $0 \leq s\leq k.$ We show that
$(\max$-$\lhd \mathfrak{U})^{Q}
\subseteq (\max $-$\lhd \mathfrak{U}^{k+1})^{Q}.$\\
Let $\Lieg \in (\max$-$\lhd \mathfrak{U})^{Q}$ and consider
\begin{center}
    $\mathcal{J}_{0} \subseteq \mathcal{J}_{1} \subseteq \mathcal{J}_{2}
    \subseteq \cdots,$ \qquad\qquad\qquad $(1)$
\end{center}
an ascending chain of $\mathfrak{U}^{k+1}$ ideals of $\Lieg$.
 Observe that for all $r \in \mathbb{N}$,
$(\mathcal{J}_{r})^{(k)} \mbox{ch} \mathcal{J}_{r} \lhd \Lieg$ and
the right multiplication operator $R_a, a \in \Lieg,$ restricted to
$\mathcal{J}_{r}$ maps $(\mathcal{J}_{r})^{(k)}$ to the zero vector.
So by Remark \ref{3.4}, we obtain that
$(\mathcal{J}_{r})^{(k)}$ is an ideal of $\Lieg.$
Moreover since $\mathcal{J}_{r} \in \mathfrak{U}^{k+1},$ one
 verifies by direct calculation that
$(\mathcal{J}_{r})^{(k)}$ is an abelian ideal of $\Lieg$ for all
$k\geq 1.$ Therefore
$$(\mathcal{J}_{0})^{(k)} \subseteq
(\mathcal{J}_{1})^{(k)} \subseteq
(\mathcal{J}_{2})^{(k)} \subseteq \cdots$$ is an ascending
chain of abelian ideals of $\Lieg$ and $\Lieg \in \max$-$\lhd
\mathfrak{U}.$ So this chain terminates. Consequently there exists
$m_{0} \in \mathbb{N}$ such that
\begin{center}
    $(\mathcal{J}_{r})^{(k)}= (\mathcal{J}_{m_{0}})^{(k)}$ for all $r \geq m_{0}$.
\end{center}
Now, consider the subchain of ideals of the initial chain (1),
$$\mathcal{J}_{m_0} \subseteq
\mathcal{J}_{m_0 + 1} \subseteq \mathcal{J}_{m_0 + 2} \subseteq
\cdots.$$ The following induced ascending chain
\begin{center}
$\frac{\mathcal{J}_{m_0}}{(\mathcal{J}_{m_{0}})^{(k)}}
\subseteq \frac{\mathcal{J}_{m_0 +
1}}{(\mathcal{J}_{m_{0}})^{(k)}} \subseteq
\frac{\mathcal{J}_{m_0 + 2}}{(\mathcal{J}_{m_{0}})^{(k)}}
\subseteq \cdots $
\end{center}
is an ascending chain of $\mathfrak{U}^{k}$ ideals of
$\frac{\Lieg}{(\mathcal{J}_{m_{0}})^{(k)}} $. Since by
induction hypothesis, $\Lieg \in (\max$-$\lhd \mathfrak{U})^{Q}
\subseteq (\max$-$\lhd \mathfrak{U}^{k})^{Q},$ this chain
terminates. So there exists $t_{0} \geq m_0 $ such that
\begin{center}
$\frac{\mathcal{J}_{t}}{(\mathcal{J}_{m_{0}})^{(k)}}=\frac{\mathcal{J}_{t_{0}}}{(\mathcal{J}_{m_{0}})^{(k)}}$
for all $t \geq t_{0} \geq m_0.$
\end{center}
Thus $\mathcal{J}_{t} = \mathcal{J}_{t_{0}},$ for all $t \geq
t_{0}$. This implies that the initial chain (1):
\begin{center}
    $\mathcal{J}_{0} \subseteq \mathcal{J}_{1} \subseteq \mathcal{J}_{2} \subseteq \cdots \subseteq \mathcal{J}_{m_{0}} \subseteq
     \mathcal{J}_{m_{0}+1} \subseteq \cdots \subseteq \mathcal{J}_{t}=\mathcal{J}_{t_{0}}=\mathcal{J}_{t_{0}+1}= \cdots$
\end{center}
terminates and therefore $(\max$-$\lhd \mathfrak{U})^{Q} \subseteq
(\max$-$\lhd \mathfrak{U}^{k})^{Q}$ for all $k\in \mathbb{N}.$
\end{proof}

\indent The following is a characterization of
Noetherian Leibniz algebras in terms of two subclasses of
Leibniz algebras, namely $ q\max$-$\lhd$ and $ (\max $-$\lhd
\mathfrak{U})^{Q}. $
\begin{Th} \label{5.2} A quasi-Noetherian Leibniz algebra is a Noetherian Leibniz algebra if and only if every quotient
algebra satisfies the maximal condition for abelian ideals. Symbolically,
\begin{center}
$\max$-$ \lhd = (\max$-$\lhd \mathfrak{U})^{Q} \cap (q \max $-$ \lhd).$\end{center}
\end{Th}
\begin{proof}

    Let $\Lieg \in \max$-$\lhd.$ We show that $\Lieg \in (\max$-$\lhd \mathfrak{U})^{Q} \cap (q \max$-$\lhd )$.
    Consider
  $I_{0} \subseteq I_{1} \subseteq I_{2} \subseteq \cdots,$
 an ascending chain satisfying the maximal condition of ideal of $\Lieg$. There exists $m_1 \in \mathbb{N}$ such that $I_{m}=I_{m_1}$,
    for all $m \geq m_1$.
But $I_{m}=I_{m_1},$ for all $m \geq m_1$ implies $\underset{i \in
\mathbb{N}}{\cup} I_{i}= I_{m_1}$, thus
$[\Lieg^{(m_{1})},\underset{i \in \mathbb{N}}{\cup}
I_{i}]=[\Lieg^{(m_{1})},I_{m_1}]
\subseteq I_{m_{1}}$ and $[\underset{i \in \mathbb{N}}{\cup} I_{i},\Lieg^{(m_{1})}]=[I_{m_1},\Lieg^{(m_{1})}] \subseteq I_{m_{1}}$.
Therefore $\Lieg \in q \max$-$\lhd$.\\
    Consider on the other hand, the ascending chain
\begin{center}
        $\frac{I_{0}}{I} \subseteq \frac{I_{1}}{I} \subseteq \cdots,$ \qquad\qquad\qquad $(7)$
    \end{center}

 of $\mathfrak{U}$ ideals of $\Lieg / I,$ for some ideal $I$ of
    $\Lieg$.\\
   From this chain, we obtain an ascending chain
    \begin{center}
        $I_{0} \subseteq I_{1} \subseteq \cdots$ \qquad\qquad\qquad $(7.1)$
    \end{center}
    of ideals of $\Lieg$. Since $\Lieg \in \max$-$\lhd$, then the chain $(7.1)$ terminates, that is, there exists $n_0 \in \mathbb{N}$ such that $I_{n}=I_{n_0}$,
    for all $n \geq n_0$. Thus, $\frac{I_{n}}{I}=\frac{I_{n_0}}{I}$, for all $n \geq n_0$.
    So, the chain $(7)$ terminates. Thus $\Lieg / I \in \max$-$\lhd \mathfrak{U}$ for some $I \lhd \Lieg$ and
    so $\Lieg \in (\max$-$\lhd \mathfrak{U})^{\mathbb{Q}}$. Hence $\Lieg \in (\max$-$\lhd \mathfrak{U})^{\mathbb{Q}} \cap (q
    \max$-$\lhd)$.\\
  Conversely, let $\Lieg \in (\max$-$\lhd \mathfrak{U})^{Q} \cap (q
    \max$-$\lhd)$ and consider the increasing chain
    \begin{center}
    $I_{0} \subseteq I_{1} \subseteq I_{2} \subseteq \cdots$ \qquad\qquad\qquad $(8)$
    \end{center}
   of ideals of $\Lieg$. Set
    $I=\underset{n}{\cup}I_{n}.$
Since $\Lieg$ is quasi-Noetherian, there exists $m\in
\mathbb{N}$ such that $[\Lieg^{(m)},I] \subseteq I_{m}$ and
$[I,\Lieg^{(m)}] \subseteq I_{m}.$ Moreover, we have:
    \begin{center}
        $I^{(m+1)}=[I^{(m)},I^{(m)}]\subseteq [\Lieg^{(m)},I^{(m)}] \subseteq [\Lieg^{(m)},I]
        \subseteq I_{m}$.
    \end{center}
    Thus $I^{(m+1)} \subseteq I_{m},$ and so $\frac{I}{I_{m}} \subseteq
    \frac{I}{I^{(m+1)}}$.
    Therefore $ (\frac{I}{I_{m}})^{(k)} \subseteq (\frac{I}{I^{m+1}})^{(k)}= \frac{I^{(k)}}{I^{(m+1)}}$.
     Taking $ k=m+1 $ we obtain $ \frac{I^{(m+1)}}{I_{m}} \subseteq \frac{I^{(m+1)}}{I^{(m+1)}}=0$.
     Hence $(\frac{I}{I_{m}}) \in \mathfrak{U}^{m+1}.$ That is, $(\frac{I}{I_{m}})^{(m+1)}=0.$
     Consequently,
    \begin{center}
                $\frac{I_{m}}{I_{m}} \subseteq \frac{I_{m+1}}{I_{m}} \subseteq \frac{I_{m+2}}{I_{m}} \subseteq
                \cdots$ \qquad\qquad\qquad $(9)$
    \end{center}
    is an ascending chain of $\mathfrak{U}^{m+1}$ ideals of $\Lieg / I_{m}$.
    Since $\Lieg \in (\max$-$\lhd \mathfrak{U})^{Q}$, by Lemma \ref{5.1}, we have
$ \Lieg \in (\max$-$\lhd \mathfrak{U}^{m+1})^{Q}$  and so $\Lieg
/ I_{m} \in \max$-$\lhd \mathfrak{U}^{m+1}.$ Therefore $(9)$
terminates, and so does $(8)$. Hence $\Lieg \in \max$-$\lhd$.
\end{proof}

\begin{Co}
    Suppose $ \mathfrak{X} $ is a class of Leibniz algebras. If $ \mathfrak{X} $ is Q-closed and
 $ \mathfrak{X} \subseteq \max$-$\lhd \mathfrak{U},$ then  $ q \max $-$ \lhd \cap \mathfrak{X} \subseteq \max $-$\lhd. $

\end{Co}
\begin{proof}
    Since $\mathfrak{X}$ is $Q$-closed, we have $\mathfrak{X} \subseteq ( \max$-$\lhd \mathfrak{U})^{Q}$.\\
     Hence, $q \max$-$\lhd \cap \mathfrak{X} \subseteq q \max$-$\lhd \cap (
\max$-$\lhd \mathfrak{U})^{Q}= \max$-$\lhd$ by Theorem \ref{5.2}. Thus $q
\max$-$\lhd \cap \mathfrak{X} \subseteq \max$-$\lhd.$
\end{proof}



\begin{thebibliography}{7}

    \bibitem{DDD}{F. A. M. Aldosray, On left quasi Noetherian rings, Internat. J. Sci. Innov. Math. Research 2 (2014) 361-365.}








    \bibitem{A. A. O.}{ S. Albeverio, Sh. A. Ayupov and B. A. Omirov, On
nilpotent and simple Leibniz algebras, Communications in Algebra,
{\bf 33}: 159-172, (2005)}.
    \bibitem{bl} Blokh, A.: A generalization of the concept of a Lie algebra, { Dokl. Akad. Nauk SSSR} {\bf 165}(3), 471--473 (1965).
   \bibitem{BT} {G. R. Biyogmam and C. Tcheka. A note on outer derivations of Leibniz algebras, Communication in Algebra, {\bf 49} (5) (2021),
   2190{2198}.}

     \bibitem{BPP} {A. N. Blagovisnaya, O. A. Pikhtilkova and S. A. Pikhtilkov, On the M.V. Zaicev problem for a Noetherian special Lie algebras, Izv. Vyssh. Uchebn. Zaved. Mat., 2017, no. 5, 26-31; Russian Math. (Iz. VUZ), 61:5 (2017), 21-25}.

    \bibitem{Falih A. M. Aldosray} {Falih A. M. Aldosray and I. Stewart. A generalized Noetherian condition for Lie algebras, Journal of Algebras and its applications(2019)1950146; DOI: 10.1142/SO219498819501469.
\bibitem{C. C.}{C. Cuvier, Alg\`ebres de Leibniz: d\'efinitions et propri\'et\'es, Annales scientifique de E. N. S, $4^{\grave{e}}$ s\'erie, tome 27, $n^{0}1,$ (1994), p1-45.}


    \bibitem{D. K. M.}{Ismail Demir, Kailash C. Misra and Ernie Stitzinger, On some
structures of Leibniz algebras, Contemporary Mathematics}.
    \bibitem{KTB} {A. Kamga Dayo, C. Tcheka and G. R. Biyogmam.} {\it $T_{Lie}$-Leibniz algebra and related
    properties,} preprint.
    }
  \bibitem{Lo 1} {J.-L. Loday}: {\it Cyclic homology}, Grundl. Math. Wiss. Bd. {\bf 301}, Springer, Berlin (1992).

    \bibitem{CCC} {J.-L. Loday}: {\it Une version non commutative des alg\`ebres de Lie: les alg\`ebres de Leibniz}, {Enseign. Math.} {\bf 39} (1993), 269--292.

    \bibitem{ASN}{Falih A. M. Aldosray and I. Stewart, Quasi Noetherian
    Lie Algebras Correction and Further Results, 2022 preprint}.

\end{thebibliography}
\end{document}